\documentclass[12pt]{article}
\usepackage{amssymb,amsmath,amsthm,bm,citesort,dsfont}
\usepackage{url,color,units}
\usepackage[margin=1in]{geometry}

\DeclareMathOperator{\Cov}{Cov}

\newcommand{\Z}{\mathds{Z}}

\newcommand{\R}{\mathds{R}}

\newcommand{\be}{\begin{equation}}
\newcommand{\ee}{\end{equation}}

\newcommand{\lip}{\text{\rm Lip}}

\renewcommand{\P}{\mathds{P}}
\newcommand{\E}{\mathds{E}}
\newcommand{\F}{\mathcal{F}}

\newcommand{\1}{\mathds{1}}
\renewcommand{\d}{{\rm d}}

\newcommand{\e}{{\rm e}}
\renewcommand{\geq}{\geqslant}
\renewcommand{\leq}{\leqslant}
\renewcommand{\ge}{\geqslant}
\renewcommand{\le}{\leqslant}

\author{Le Chen\\University of Kansas
	\and Davar Khoshnevisan\\University of Utah
	\and Kunwoo Kim\\MSRI-Berkeley}
	\title{{ A Boundedness Trichotomy\\for 
	the Stochastic Heat Equation}\thanks{
	Research supported in part by grants from
	the Swiss Federal Fellowship Program (P2ELP2\_151796, L.C.)
	and the United States' National Science Foundation  (DMS-1307470, D.K.;
	0932078000, K.K. through The
	Mathematical Sciences Research Institute at UC Berkeley)}}
\date{%Last update: 
	October 12, 2015}%April 26, 2015}%March 22, 2015}%October 22, 2014}
\newtheorem{stat}{Statement}[section]

\newtheorem{theorem}[stat]{Theorem}
\newtheorem{lemma}[stat]{Lemma}
\theoremstyle{definition}

\numberwithin{equation}{section}
\begin{document}
\maketitle
\begin{abstract}
	We consider the stochastic heat equation with a multiplicative
	white noise forcing term under standard ``intermitency
	conditions.'' The main finding of this paper is that, under mild
	regularity hypotheses, the 
	 a.s.-boundedness of the solution
	$x\mapsto u(t\,,x)$ can be characterized generically by the decay
	rate,  at $\pm\infty$,
	of the initial function $u_0$. More specifically,
	we prove that there are 3 generic boundedness regimes, depending
	on the numerical value of
	$\bm\Lambda := \lim_{|x|\to\infty} |\log u_0(x)|/(\log|x|)^{2/3}$.\\


\end{abstract}%\newpage

\section{Introduction}
It has been recently shown \cite{CJK} that a large family 
of parabolic stochastic PDEs are chaotic in the sense that small changes
in their initial value can lead to drastic changes in the global
structure of the solution. In this paper we describe 
some of the quantitative aspects of the nature of
that chaos. 

Consider the solution $u=\{u(t\,,x)\}_{t>0,x\in\R}$ of  the stochastic initial-value problem
\begin{equation}\label{SHE}\left[\begin{split}
	&\dot{u}(t\,,x)= \frac12 u''(t\,,x) +\sigma(u(t\,,x))\,\xi(t\,,x)
		&[t>0,\, x\in\R],\\
	&\text{subject to }u(0\,,x)= u_0(x)&[x\in\R],
\end{split}\right.\end{equation} 
where $\xi$ denotes space-time white noise; that is, a centered Gaussian
random field with covariance functional 
\[
	\Cov[\xi(t\,,x)\,,\xi(s\,,y)]=\delta_0(s-t)
	\delta_0(x-y)
	\qquad[s,t\ge0,\, x,y\in\R].
\]
Alternatively, we can construct $\xi$ as
$\xi(t\,,x) = \partial^2_{x,t} B(t\,,x),$
in the sense of distributions, where $B:=\{B(t\,,x)\}$ is a mean-zero
continuous Gaussian process with covariance
\[
	\text{\rm Cov}[ B(t\,,x) \,, B(s\,,y) ] = \min(s\,,t) \times\min(|x|\,,|y|)
	\times \1_{(0,\infty)}(xy),
\]
for all $s,t\ge0$ and $x,y\in\R$
[$B$ is known as a space-time Brownian sheet.]

Some of the commonly-used assumptions on the initial value $u_0$ 
and the nonlinearity $\sigma$
are that:
\begin{enumerate}
	\item[(a)] $u_0\in L^\infty(\R)$ is non random; $u_0(x)\ge 0$ for almost all 
		$x\in\R$; and $u_0>0$ on a set of positive Lebesgue measure;
		and 
	\item[(b)] $\sigma:\R\to\R$ is Lipschitz continuous and nonrandom.
\end{enumerate}
These conditions will be in place from now on.
Under these conditions, it is well known \cite{Dalang,cbms,Walsh}
that \eqref{SHE} admits a continuous predictable solution
$u$ that is uniquely defined via the \emph{a priori} condition,
\[
	\sup_{t\in[0,T]}\sup_{x\in\R}\E\left(|u(t\,,x)|^k\right)<\infty
	\quad\text{for all $T>0$ and $k\ge 2$}.
\]

Throughout, we will suppose, in addition, that the nonlinearity $\sigma$ satisfies
\begin{equation}\label{sig}
	\sigma(0)=0\quad
	\text{ and }\quad
	{\rm L}_\sigma:=\inf_{w\in\R}|\sigma(w)/w|>0.
\end{equation}
The first condition in \eqref{sig} implies that there exists a $\P$-null set off which
\[
	u(t\,,x)>0\qquad\text{for all $t>0$ and
	$x\in\R$;}
\]
see Mueller \cite{Mueller1,Mueller2}. And the  condition on
the positivity of ${\rm L}_\sigma$ is an ``intermittency condition,''
and implies among other things that 
the moments of $u(t\,,x)$ grow exponentially with time \cite{FK:09}.
The intermittency condition arose earlier in the work of Shiga \cite{Shiga}
on interacting infinite systems of It\^o-type diffusion processes.

Together, the two conditions in \eqref{sig} suffice to ensure that the solution $u$ to
\eqref{SHE} is ``chaotic'' in the sense that its global behavior, at all times, depends 
strongly on its initial state $u_0$. To be more concrete, we know for example
that if $\inf_{x\in\R} u_0(x)\ge \varepsilon$ for a constant $\varepsilon>0$,  
then \eqref{sig} implies that
\[
	\P\left\{ \sup_{x\in\R} u(t\,,x) =\infty\right\}=1\qquad\text{for
	all $t>0$};
\]
see \cite{CJK}. And by contrast, $\sup_{x\in\R}u(t\,,x)<\infty$ a.s.\
for all $t>0$ when $u_0$ is Lipschitz continuous [say] with compact support; see
\cite{FK}.
Based on these results, 
one can imagine that if and when $u_0(x)\to 0$ as $|x|\to\infty$, then 
$\sup_{x\in\R} u(t\,,x)$ can be finite or infinite for some or even all $t$, 
depending on the nature of the decay of $u_0$ at $\pm\infty$. 
The goal of this article is to describe precisely the amount of decay $u_0$ needs 
in order to ensure that $u(t\,,\cdot)$ is a bounded function almost surely. 
Because we are interested in almost-sure finiteness of the global maximum
of the solution, this undertaking
is different in style, as well as in methodology, from results that describe 
stochastic PDEs for which the spatial maximum of the solution is
in $L^k(\P)$ for some $1\le k<\infty$ \cite{FK,KonstantinosGerencer}.

We will make additional simplifying assumptions on the function $u_0$
in order to make our derivations as non-technical as possible, yet
good enough to describe the new phenomenon that we plan to present.
In view of this, we will assume throughout that 
\[
	\lim_{z\to\infty} u_0(z)=0,\quad
	u_0(x)=u_0(-x),\quad\text{and}\quad
	u_0(x) \ge u_0(y)\quad\text{if $0\le x\le y$}.
\]
Finally, we assume that the following limit exists:
\[
	\bm{\Lambda} := \lim_{|x|\to\infty} \frac{|\log u_0(x)|}{(\log |x|)^{2/3}}.
\]
The existence of  this limit is a mild condition,
since $\bm\Lambda$ can be any number in the 
\emph{closed} interval $[0\,,\infty]$.
 
 Throughout, define 
 \[
 	M(t) := \sup_{x\in\R} u(t\,,x)\hskip1in[t>0].
 \]
The following trichotomy is the main finding of this paper.

\begin{theorem}\label{th:0-1}
	Under the preceding conditions:
	\begin{enumerate}
		\item If $\bm{\Lambda}=\infty$, then 
			$\P \{ M(t)<\infty\text{ for all $t>0$} \}=1;$
		\item If $\bm{\Lambda}=0$, then
			$\P \{ M(t)=\infty\text{ for all $t>0$} \}=1;$%
			\footnote{%
				Of course, $t=0$ is different from $t>0$
				since $M(0)<\infty$ in all cases.
			}
		\item If $0<\bm{\Lambda}<\infty$,
			then there exists a random variable $\mathcal{T}$ and two nonrandom
			constants $t_1,t_2\in(0\,,\infty)$ such that: (i) $t_1< \mathcal{T}< t_2$ a.s.;
			and (ii)
			\[
				\P\left\{  M(t)<\infty \,\,\forall t<\mathcal{T}
				\quad\text{and}\quad M(t)=\infty\,\,\forall
				t>\mathcal{T} \right\}=1.
			\]
	\end{enumerate}
\end{theorem}

From now on we find it more convenient to write the solution to
\eqref{SHE}, using more standard probability notation, as
\[
	u_t(x) := u(t\,,x)\qquad\text{for all $t>0$ and $x\in\R$.}
\]
In particular, $u_t$ does not refer to the time derivative of $u$.

We also denote the Lipschitz constant of $\sigma$ by
\[
	\lip_\sigma:=\sup_{-\infty<x\neq y<\infty}
	\left| \frac{\sigma(x)-\sigma(y)}{x-y}\right|.
\]

\section{Tail Probabilities via Insensitivity Analysis}%\label{sec:suscept}

One of the first problems that we need to address is  itself
related  to matters
of chaos, and more specifically to the problem of how sensitive
the solution of \eqref{SHE} is to ``small'' changes in the initial function.
A suitable solution to this sensitivity problem has a number of
interesting consequences. In the present context, we
will use sensitivity analysis to
derive sharp  estimates for the tail of the distribution of 
the solution $u_t(x)$ to \eqref{SHE}. 

We will have to interpret our sensitivity problem in a rather specific way,
which we would like to describe in terms of an adversarial game
between a player [Player 1] and Mother Nature [Player 2]. 

In this game, both players know the values of the external
noise $\xi$. Player 2 knows also the initial function $u_0$, and hence
the solution $u_t(x)$ at all space-time points $(t\,,x)$. Player 1, on the other hand
knows the values of $u_0(x)$ only for $x$ in some pre-determined interval
$[a-r\,,a+r]$. Player 1 guesses that the initial function is $v_0$,
in some fashion or another,
where $v_0(x) = u_0(x)$ for all $x\in[a-r\,,a+r]$. 

Let $v_t(x)$ denote the solution
to \eqref{SHE} with initial values $v_0$; the function $v$ is Player 1's guess
for the solution to \eqref{SHE}. The following shows that if $t\ll r$, then
near the middle portion of the spatial interval
$[a-r\,,a+r]$, the solution appears essentially the same to both Players 1 and  2.
This shows that, for a long time $[t\ll r]$,
the values of the solution to \eqref{SHE} in the middle portion of
$[a-r\,,a+r]$ are insensitive to basically all possible changes to the initial
value outside of $[a-r\,,a+r]$ .

\begin{theorem}\label{th:suscept}
	Choose and fix two parameters $a\in\R$ and $r>0$.
	Let $u$ and $v$  denote the solutions
	to \eqref{SHE} with respective initial values $u_0$
	and $v_0$, where $u_0,v_0\in L^\infty(\R)$ are
	nonrandom and $u_0(x)=v_0(x)$ a.e.\ on
	$[a-r\,,a+r]$. Then, for all $t>0$,
	\begin{equation}\label{eq:suscept}
		\sup_{|x-a|\le r/4}
		\E\left(|u_t(x) - v_t(x)|^2\right) \le C\ell
		\|u_0-v_0\|_{L^\infty(\R)}^2 \exp\left( -\frac{r^2}{16t} +
		\frac{\lip_\sigma^4 t}{4}\right),
	\end{equation}
	where $\ell:=1+\lip_\sigma^4$ and $C := 96[1\vee\lip_\sigma^{-4}].$
\end{theorem}
It might help to emphasize that $t\ll r$ if and only if the exponential on 
the right-hand side of \eqref{eq:suscept} is a  small quantity.

The proof of Theorem \ref{th:suscept} relies on two technical lemmas.
The first lemma is an elementary fact about the linear 1-dimensional
heat equation [and associated convolution equations].

\begin{lemma}\label{lem:suscept}
	Suppose $h\in L^\infty(\R)$ is a nonrandom function
	that is equal to zero a.e.\ in an interval $[a-r\,,a+r]$. Then,
	for all $t>0$,
	\[
		\sup_{x:|x-a|\le r/2} \left| (p_t*h)(x)  \right| \le 2\|h\|_{L^\infty(\R)}\cdot
		\e^{-r^2/(8t)}.
	\]
\end{lemma}

\begin{proof}
	By Minkowski's inequality,
	\[
		|(p_t*h)(x) |\le\|h\|_{L^\infty(\R)}\cdot\int_{|w+x-a|>r} p_t(w)\,\d w.
	\]
	If $|w+x-a|>r$ and $|x-a|\le r/2$, then certainly $|w|>r/2$. 
	The lemma follows from the simple bound,
	$\int_{|w|>r/2} p_t(w)\,\d w\le 2\exp\{ -
	r^2/(8t)\}$.
\end{proof}

In order to introduce the second lemma we first need some notation.
Let ``$\odot$'' denote space-time convolution. That is,
for all measurable space-time functions $f$ and $g$,
\[
	(f\odot g)_t(x) := \int_0^t\d s\int_{-\infty}^\infty \d y\
	f_{t-s}(x-y) g_s(y), 
\]
pointwise, whenever the [Lebesgue] integral is absolutely convergent.
For every $\alpha>0$, consider the space-time kernel 
$\mathcal{K}^{(\alpha)}$, defined as
\begin{equation}\label{K}
	\mathcal{K}_t^{(\alpha)}(x) := \frac{\alpha^2}{2}
	p_{t/2}(x)\left[\frac{1}{\sqrt{\pi t}} + 
	\alpha^2 \exp\left(\frac{\alpha^4 t}{4}\right)
	\Phi\left(\alpha^2\sqrt{\frac t2}\right)\right],
\end{equation}
for all $t>0$ and $x\in\R$, where
$\Phi(x) := (2\pi)^{-1/2} \int_{-\infty}^x\exp(-y^2/2)\,\d y$
$[x\in\R]$ denotes the cumulative distribution function of a standard normal law on 
the line. We can now state our second technical lemma.

\begin{lemma}\label{lem:CD}
	Choose and fix a deterministic function $f\in L^\infty(\R)$,
	and define a space-time function $\mathcal{J}$ via
	$\mathcal{J}_t(x):=(p_t*f)(x)$ for all $t>0$ and $x\in\R$. 
	Suppose $(t\,,x)\mapsto F_t(x)$ is a measurable
	space-time function that is bounded in $x$ and grows at most
	exponentially in $t$, and satisfies
	\begin{equation}\label{RE}
		F \le \mathcal{J}^2  + \alpha^2 (F\odot p^2),
	\end{equation}
	pointwise for a fixed constant $\alpha>0$. Then,
	\begin{equation}\label{FF}
		F \le \mathcal{J}^2 + \left( \mathcal{J}^2\odot \mathcal{K}^{(\alpha)}\right)
		\qquad\text{\rm pointwise}.
	\end{equation}
\end{lemma}
%\textcolor{red}{
%\ref{??} For future references, shall we state the lemma slightly more general?
%The proof will remain the same.
%\begin{lemma}\label{lem:CD}
%	Choose and fix a nonnegative measure $\mu$ on $\R$ such that
%	$\mathcal{J}_t(x):=(p_t*\mu)(x)<\infty$ for all $t>0$ and $x\in\R$.
%	Suppose $(t\,,x)\mapsto F_t(x)$ is a measurable
%	space-time function that is bounded in $x$ and grows at most
%	exponentially in $t$, and satisfies
%	\begin{equation}\label{RE}
%		F \le \mathcal{J}^2  + \alpha^2 (F\odot p^2),
%	\end{equation}
%	pointwise for a fixed constant $\alpha>0$. Then,
%	\begin{equation}\label{FF}
%		F \le \mathcal{J}^2 + \left( \mathcal{J}^2\odot \mathcal{K}^{(\alpha)}\right)
%		\qquad\text{\rm pointwise}.
%	\end{equation}
%\end{lemma}
%}

\begin{proof}[Proof (sketch)]
	This is basically the first part of eq.\ (2.21) of Chen and Dalang \cite{ChenDalang},
	but is stated here in slightly more general terms. Therefore, we skip
	the details and merely point out
	how one can relate Lemma \ref{lem:CD} to the work of
	Chen and Dalang \cite{ChenDalang}, deferring the details to the latter reference.
	
	In order to see how one can deduce
	this lemma from the arguments of Chen and Dalang, let us consider the
	stochastic heat equation \eqref{SHE} with a nonrandom initial value
	$f$, and let $U$ denote the solution. We can write the solution in integral
	form as follows:
	\[
		U_t(x) = \mathcal{J}_t(x) + \int_{(0,t)\times\R} p_{t-s}(y-x)
		\sigma(U_s(y))\,\xi(\d s\,\d y).
	\]
	Elementary properties of the stochastic integral imply that
	\begin{align*}
		\E\left(|U_t(x)|^2\right) &= \mathcal{J}_t^2(x) +
			\int_0^t\d s\int_{-\infty}^\infty\d y\ [p_{t-s}(y-x)]^2
			\E\left( |\sigma(U_s(y))|^2\right)\\
		&\le \mathcal{J}_t^2(x) +\lip_\sigma^2
			\int_0^t\d s\int_{-\infty}^\infty\d y\ [p_{t-s}(y-x)]^2
			\E\left( |U_s(y)|^2\right).
	\end{align*}
	That is, in the special case that $F_t(x)=\E(|U_t(x)|^2)$
	satisfies \eqref{RE} with $\alpha=\lip_\sigma$.
	In this special case, Theorem 2.4 of Chen and Dalang \cite{ChenDalang}
	implies \eqref{FF}, and our function
	$\mathcal{K}^{(\lip_\sigma)}$ coincides with their function
	$\overline{\mathcal{K}}$. For general $F$ and $\alpha$, the very
	same proof works equally well. 
\end{proof}

\begin{proof}[Proof of Theorem \ref{th:suscept}]
	We begin by writing $u$ and $v$ in integral form as follows:
	\begin{align*}
		u_t(x) &= (p_t*u_0)(x) + \int_{(0,t)\times\R} p_{t-s}(y-x)
			\sigma(u_s(y))\,\xi(\d s\,\d y),\\
		v_t(x) &= (p_t*v_0)(x) + \int_{(0,t)\times\R} p_{t-s}(y-x)
			\sigma(v_s(y))\,\xi(\d s\,\d y).
	\end{align*}
	Define
	\[
		f(x) := |u_0(x)-v_0(x)|\qquad\text{\rm for all $x\in\R$},
	\]
	and set $\mathcal{J}_t(x) := (p_t*f)(x)$ for all $t>0$ and $x\in\R$.
	Then clearly,
	\begin{align*}
		\E\left(|u_t(x)-v_t(x)|^2\right)
			&\le \left| \mathcal{J}_t(x)\right|^2
			+ \int_0^t\d s\int_{-\infty}^\infty\d y\
			[p_{t-s}(y-x)]^2\E\left(\left|\sigma(u_s(y))-\sigma(v_s(y))\right|^2\right)\\
		&\le\left| \mathcal{J}_t(x)\right|^2
			+ \lip_\sigma^2\int_0^t\d s\int_{-\infty}^\infty\d y\
			[p_{t-s}(y-x)]^2\E\left(\left|u_s(y)-v_s(y)\right|^2\right).
	\end{align*}
	In other words, the space-time function,
	\[
		F_t(x) := \E\left( |u_t(x)-v_t(x)|^2\right)
		\qquad[t>0,x\in\R],
	\]
	satisfies \eqref{RE} with $\alpha=\lip_\sigma$. Therefore,
	\eqref{FF} implies that
	\begin{equation}\label{FFF}
		F \le \mathcal{J}^2 + \left( \mathcal{J}^2\odot \mathcal{K}^{(\lip_\sigma)}\right)
		\qquad\text{\rm pointwise}.
	\end{equation}
	
	An inspection of the function $\mathcal{K}^{(\lip_\sigma)}$---see
	\eqref{K}---shows that
	\[
		\mathcal{K}^{(\lip_\sigma)}_t(x) \le \ell p_{t/2}(x)\left[ \frac{1}{\sqrt t} + 
		\exp\left(\frac{\lip_\sigma^4 t}{4}\right)\right],
	\]
	 for all $t>0$ and $x\in\R$. Consequently,
	\[
		\left(\mathcal{J}^2\odot \mathcal{K}^{(\lip_\sigma)}\right)_t(x) \le 
		\ell\int_0^t \left( \mathcal{J}^2_{t-s}*p_{s/2}\right)(x)
		\left[ \frac{1}{\sqrt s}+\exp\left(\frac{\lip_\sigma^4 s}{4}\right)\right] \d s.
	\]
	Set $B:=\|u_0-v_0\|_{L^\infty(\R)}$, and observe that
	$\|\mathcal{J}_{t-s}\|_{L^\infty(\R)}\le B$. According to Lemma \ref{lem:suscept},
	\begin{equation}\label{JF}
		\sup_{|y-a|\le r/2}\mathcal{J}_{t-s}(y) \le 2 B \exp\left(-\frac{r^2}{8(t-s)}\right)\le
		2B\e^{-r^2/(8t)}.
	\end{equation}
	Consequently, we can split up the
	ensuing integral into regions where $|y-a|\le r/2$ and
	where $|y-a|>r/2$ in order to see that
	\begin{align*}
		\left( \mathcal{J}^2_{t-s}*p_{s/2}\right)(x) 
			& \le 4 B^2\e^{-r^2/(4t)} + B^2\int_{|y+x-a|>r/2}
			p_{s/2}(y)\,\d y\\
		&\le 4 B^2\e^{-r^2/(4t)} + B^2\int_{|y|>r/4}
			p_{s/2}(y)\,\d y,
	\end{align*}
	uniformly for all $|x-a|\le r/4$ and $0<s<t$. This and a simple tail bound together yield
	\[
		\sup_{|x-a|\le r/4} \left( \mathcal{J}^2_{t-s}*p_{s/2}\right)(x) 
		\le 4 B^2\e^{-r^2/(4t)} + 2B^2\e^{-r^2/(16s)}
		\le 6B^2\e^{-r^2/(16t)},
	\]
	for all $0<s<t$.
	Thus, we can see that, uniformly for all $t>0$ and all $x$ that
	satisfy $|x-a|\le r/4$,
	\begin{align*}
		\left(\mathcal{J}^2\odot \mathcal{K}^{(\lip_\sigma)}\right)_t(x) &\le  
			6\ell B^2\e^{-r^2/(16t)}\cdot \int_0^t 
			\left[ \frac{1}{\sqrt s}+\exp\left(\frac{\lip_\sigma^4 s}{4}\right)\right] \d s\\
		&\le 6\ell B^2 \e^{-r^2/(16t)}\cdot\left[2\sqrt t + 
			\frac{4}{\lip_\sigma^4}\exp\left(\frac{\lip_\sigma^4 t}{4}\right)\right]\\
		&\le 6\ell B^2\exp\left( -\frac{r^2}{16t} + 
			\frac{\lip_\sigma^4 t}{4}\right)\cdot
			\sup_{s>0}\left[2\sqrt s\, \exp\left(-\frac{\lip_\sigma^4 s}{4}\right) + 
			\frac{4}{\lip_\sigma^4}\right]\\
		&\le 48\ell 
			\left[1\vee \frac{1}{\lip_\sigma^4}\right]
			B^2\exp\left( -\frac{r^2}{16t} + \frac{\lip_\sigma^4 t}{4}\right);
	\end{align*}
	consult also \eqref{FFF}. 
	Combine this estimate with \eqref{JF}
	and \eqref{FFF} to finish.
\end{proof}

Our two technical Lemmas \ref{lem:suscept} and \ref{lem:CD} yield the
following tail probability bounds.

\begin{theorem}\label{th:tails}
	There exist universal constants $0<K,L<\infty$ such that
	for all $\varepsilon>0$,
	\[
		- \frac{L \Lambda^{3/2}}{\sqrt t}
		\le
		\liminf_{|x|\to\infty} \frac{\log \P\left\{ u_t(x)>\varepsilon\right\}}{\log |x|}
		\le
		\limsup_{|x|\to\infty} \frac{\log \P\left\{ u_t(x)>\varepsilon\right\}}{\log |x|}
		\le
		-\frac{K \Lambda^{3/2}}{\sqrt t},
	\]
	uniformly for all $t$ in every fixed compact subset of $(0\,,\infty)$.
\end{theorem}

We prove this theorem in two parts. In the first part we establish the claimed
lower bound on $\liminf_{|x|\to\infty}(\,\cdots)$.
The corresponding upper bound on $\limsup_{|x|\to\infty}(\,\cdots)$ is
derived afterward in a second part.

\begin{proof}[Proof of Theorem \ref{th:tails}: Part 1.]
	Let $u^{(0)}_t(x):=u_0(x)$ and define
	\[
		u^{(n+1)}_t(x) := (p_t*u_0)(x) + \int_{(0,t)\times\R} p_{t-s}(y-x)
		\sigma\left(u^{(n)}_s(y)\right)\xi(\d s\,\d y),
	\]
	for all $n\ge 0$, $t>0$, and $x\in\R$. It is well known that 
	$u^{(n)}_t(x)\to u_t(x)$ in $L^2(\P)$ as $n\to\infty$,
	for every $t>0$ and $x\in\R$; see Walsh \cite{Walsh}. Since
	$u^{(0)}$ and $p_t*u_0$ are symmetric functions, the symmetry
	of white noise [in law] shows that  $\{u^{(n+1)}_t(x)\}_{x\in\R}$ and
	$\{u^{(n+1)}_t(-x)\}_{x\in\R}$ have the same law for all $n\ge 0$. We
	can let $n\to\infty$ in order to deduce, in particular,
	that the random variables $u_t(x)$ and $u_t(-x)$ have the same distribution
	for each $t>0$ and $x\in\R$.
	
	In light of the preceding symmetry property,
	in order to derive the stated lower bound for
	$\P\{u_t(x)>\varepsilon\}$, it remains to prove that if $\bm{\Lambda}<\infty$, then
	\begin{equation}\label{plan:LB}
		\liminf_{x\to\infty} \frac{\log\P\left\{ u_t(x)>\varepsilon\right\}}{\log x}
		\ge - \frac{L\bm{\Lambda}^{3/2}}{\sqrt t}.
	\end{equation}
	[The assertion holds trivially when $\bm{\Lambda}=\infty$.]
	Let us consider now the case that $\bm\Lambda<\infty$.
	
	Choose and fix an arbitrary  number $a>0$, and define
	$w=\{w_t(x)\}_{t>0,x\in\R}$
	to be the solution to \eqref{SHE} with the following
	respective initial value:
	\[
		w_0(x) := u_0(|x|\vee(3a/2))
		\qquad\text{for all $x\in\R$}.
	\]
	The construction of the process $w$ does not present any problems
	because $w_0$ is a nonrandom elements of $L^\infty(\R)$; in fact,
	$0\le w_0 \le u_0$. These inequalities have the additional consequence
	that
	\begin{equation}\label{w<u}
		w_t(x) \le u_t(x)\qquad
		\text{for all $t>0$ and $x\in\R$},
	\end{equation}
	thanks to Mueller's comparison principle \cite{Mueller1,Mueller2}. 
	Therefore, it remains to find a
	lower  bound for the tails of the distribution of $w_t(x)$. 
	
	Define 
	\[
		z_0(x) := u_0(3a/2)\qquad\text{for all $x\in\R$},
	\]
	and let $z:=\{z_t(x)\}_{t>0,x\in\R}$ denote the solution to \eqref{SHE}
	with initial value $z_0$. By the comparison principle,
	$w_t(x) \le z_t(x)$ for all $t\ge0$ and $x\in\R$. We now use our susceptibilty
	estimate [Theorem  \ref{th:suscept}] in order to prove that there is
	a similar lower bound near the point $x=a$, 
	provided that we introduce a small error. 
	Specifically, we apply Theorem \ref{th:suscept} with $r:=a/2$ in
	order to see that
	\begin{align}\notag
		\sup_{x\in [7a/8,9a/8]}\E\left( \left| w_t(x) - z_t(x) \right|^2 \right)
			&\le C\ell\| w_0-z_0\|_{L^\infty(\R)}^2
			\exp\left( -\frac{a^2}{64t} + \frac{\lip_\sigma^4 t}{4}\right)\\
		&\le C\ell\|u_0\|_{L^\infty(\R)}^2
			\exp\left( -\frac{a^2}{64t} + \frac{\lip_\sigma^4 t}{4}\right).
			\label{w-z}
	\end{align}
	Since $z$ solves \eqref{SHE} with constant initial function $z_0(\cdot)\equiv u_0(3a/2)$,
	Theorems 5.5 [page 44] and 6.4 [page 57] of Ref.\
	\cite{cbms} tell us that there exists a finite
	universal constant $A>2$ such that
	\begin{equation}\label{z:mom}
		A^{-k}\left[u_0(3a/2)\right]^k \e^{k^3t/A}\le
		\E\left(|z_t(x)|^k\right) \le A^k \left[u_0(3a/2)\right]^k \e^{Ak^3t},
	\end{equation}
	simultaneously for all $x\in\R$, $t>0$, and $k\in[2\,,\infty)$. Actually,
	the results of \cite{cbms} imply
	the lower bound for $\E(|z_t(x)|^k)$ only in the case that
	$\sigma(z)=\text{const}\cdot z$ for all $z\in\R$. The general case follows
	from that fact and the moment comparison theorem of Joseph et al \cite{JKM}.
	
	In any case, we apply the Paley--Zygmund inequality, as in 
	Ref.\ \cite[Chapter 6]{cbms}, in order to see that
	\begin{align*}
		\P\left\{ z_t(x) \ge \tfrac12 A^{-1}u_0(3a/2) \e^{k^2t/A}\right\} &\ge 
			\P\left\{ z_t(x) \ge \tfrac12\|z_t(x)\|_{L^k(\P)}\right\}\\
		&\ge \tfrac14\cdot\frac{\left[ \E\left(|z_t(x)|^k\right)\right]^2}{
			\E\left(|z_t(x)|^{2k}\right)}\\
		&\ge \tfrac14  A^{-4k} \exp\left( -\left[8A-\frac{2}{A}\right]k^3t\right),
	\end{align*}
	uniformly for all real number $x\in\R$, $k\ge 2$, and $t>0$. 
	Since $A>0$, it follows that $8A-(2/A)< 8A$, and hence
	\[
		\P\left\{ z_t(x) \ge \tfrac12 A^{-1}u_0(3a/2) \e^{k^2t/A}\right\}
		\ge \tfrac14 \exp\left\{ - 8Ak^3t -4k\log A\right\},
	\]
	uniformly for all real number $x\in\R$, $k\ge 2$, and $t>0$. 
	Choose and fix an arbitrary number $\varepsilon>0$. We apply the preceding with 
	\[
		k := \sqrt{\frac A t
		\left|\log\left(\frac{4A\varepsilon}{u_0(3a/2)}\right)\right|};
	\]
	 equivalently,
	 $\tfrac12 A^{-1}u_0(3a/2) \e^{k^2t/A} = 2\varepsilon.$
	Since $u_0(3a/2)\to0$ as $a\to\infty$, it follows readily that $k\ge 2$
	if $a$ is sufficiently large [how large depends only on $A$]. Hence,
	\[
		\inf_{x\in\R}\log \P\left\{ z_t(x) \ge 2\varepsilon\right\}\\
		\ge  \exp\left( - \frac{L+o(1)}{\sqrt t}\left|\log u_0(3a/2)\right|^{3/2}\right),
	\]
	for all $a$ large, where $o(1)\to0$ as $a\to\infty$
	and $L:= 8A^{5/2}+A^{1/2}$. Note that $L$ is a universal
	constant [since $A$ is].
	
	The preceding estimate, \eqref{w<u}, and \eqref{w-z} together imply that, as $a\to\infty$,
	\begin{equation}\label{prec:LB}\begin{split}
		\P\left\{ u_t(x)\ge\varepsilon\right\} &
			\ge\P\left\{ w_t(x) \ge\varepsilon\right\}\\
		&\ge \P\left\{ z_t(x) \ge2\varepsilon\right\}
			- \P\left\{ |w_t(x)-z_t(x)|\ge\varepsilon\right\}\\
		&\ge \exp\left( - \frac{L+o(1)}{\sqrt t}\left|\log u_0(3a/2)\right|^{3/2}\right)
			- A_1 \e^{-a^2/(64t)},
	\end{split}\end{equation}
	uniformly for all $x\in [7a/8\,,9a/8]$, where $A_1<\infty$ does not
	depend on $a$. The condition $\bm{\Lambda}<\infty$ implies that
	$a^{-2}|\log u_0(3a/2)|^{3/2}\to 0$ as $a\to\infty$. Therefore,
	\eqref{prec:LB} implies \eqref{plan:LB} and hence the theorem.
\end{proof}

\begin{proof}[Proof of Theorem \ref{th:tails}: Part 2.]
	In analogy with the proof of part 1,
	it suffices to establish the following: If $\bm{\Lambda}>0$, then
	\[
		\limsup_{x\to\infty} \frac{\log\P\left\{ u_t(x)>\varepsilon\right\}}{\log x}
		\le - \frac{K\bm{\Lambda}^{3/2}}{\sqrt t}.
	\]
	[This is vacuously true when $\bm{\Lambda}=0$.]
	From here on we assume that $\bm\Lambda>0$.
	
	Choose and fix an arbitrary  number $a>0$, and define
	$w=\{w_t(x)\}_{t>0,x\in\R}$
	to be the solution to \eqref{SHE} subject to the following
	initial value:
	\[
		w_0(x) := u_0(|x|\wedge (a/2))
		\qquad\text{for all $x\in\R$}.
	\]
	The process $w$ is the present analogue of its counterpart---also dubbed
	as $w$---in Part 1 of the proof.
	As was the case in Part 1, one can construct $w$ in a standard way
	because $w_0$ is a nonrandom elements of $L^\infty(\R)$
	[$0\le u_0 \le w_0$]. Furthermore,
	\begin{equation}\label{w>u}
		w_t(x) \ge u_t(x)\qquad
		\text{for all $t>0$ and $x\in\R$},
	\end{equation}
	thanks to Mueller's comparison principle \cite{Mueller1,Mueller2}. Compare with \eqref{w<u}.
	Therefore, it remains to find an
	upper  bound for the tails of the distribution of $w_t(x)$. 

	Define
	\[
		z_0(x) := u_0(a/2)\qquad\text{for all $x\in\R$},
	\]
	and let $z:=\{z_t(x)\}_{t>0,x\in\R}$ denote the solution to \eqref{SHE}
	with initial value $z_0$. By the comparison principle,
	$w_t(x) \ge z_t(x)$ for all $t\ge0$ and $x\in\R$. And now use our susceptibilty
	estimate [Theorem  \ref{th:suscept}] in  analogy with the proof of
	part 1 of the theorem in order to see that
	\begin{equation}\label{w-z:1}
		\E\left( \left| w_t(x) - z_t(x) \right|^2 \right)
		\le C\ell\|u_0\|_{L^\infty(\R)}^2
		\exp\left( -\frac{a^2}{64t} + \frac{\lip_\sigma^4 t}{4}\right),
	\end{equation}
	uniformly for all $x\in [7a/8\,,9a/8]$.
	[Compare with \eqref{w-z}.] Now we relabel \eqref{z:mom} $[a/2\leftrightarrow
	3a/2]$ to see that, for the same constant $A$ that appeared in \eqref{z:mom},
	\[
		\E\left(|z_t(x)|^k\right) \le [Au_0(a/2)]^k \e^{Ak^3t},
	\]
	simultaneously for all $x\in\R$, $t>0$, and $k\in[2\,,\infty)$. 
	Chebyshev's inequality yields
	\begin{align*}
		\P\{z_t(x)\ge \varepsilon/2\} &\le 
			\inf_{k\ge 2}\left[\frac{2Au_0(a/2)\e^{Ak^2t}}{\varepsilon}\right]^k\\
		&= \exp\left( -\frac{2}{3\sqrt{3At}}\left[
			\log\left(\frac{\varepsilon}{2Au_0(a/2)}\right)\right]^{3/2}\right),
	\end{align*}
	uniformly for every real number $x$.
	This, \eqref{w>u}, and \eqref{w-z:1} together imply that
	\begin{align*}
		\P\left\{ u_t(x)\ge\varepsilon\right\}&\le \P\left\{ w_t(x)\ge\varepsilon\right\}
			\le \P\{z_t(x)\ge \varepsilon/2\}
			+ \P\left\{ |w_t(x)-z_t(x)| \ge\varepsilon/2\right\}\\
		&\le \exp\left( -\frac{2}{3\sqrt{3At}}\left[
			\log\left(\frac{\varepsilon}{2Au_0(a/2)}\right)\right]^{3/2}\right)
			+ A_1\e^{-a^2/(64t)},
	\end{align*}
	uniformly for all $x\in [7a/8\,,9a/8]$, where $A_1$ is a finite constant
	that does not depend on $a$. Part 2 can be deduced easily from this estimate.
\end{proof}

\section{Proof of Theorem \ref{th:0-1}}
We will soon see  that, in order to prove Theorem \ref{th:0-1}
it suffices to consider separately the cases that $\bm{\Lambda}>0$
and $\bm{\Lambda}<\infty$. [There is, of course, some overlap between
the two cases.]
The two portions require different ideas; let us begin with the case
$\bm{\Lambda}>0$, since the proof is uncomplicated and can be carried
out swiftly.

\subsection{Part 1 of the Proof}
Throughout this part of the proof, we assume that
\[
	\bm{\Lambda}>0,
\]
keeping in mind that $\bm{\Lambda}=\infty$ is permissible, as a particular case.

Choose and fix a [finite] number
\[
	\lambda\in(0\,,\bm{\Lambda}),
\]
and introduce two new parameters $\tau$ and $T$ as follows:
\[
	0<\tau<T := \frac{K^2\lambda^3}{64},
\]
where $K$ is the universal constant that appeared in the
statement of Theorem \ref{th:tails}.
We plan to prove that
\begin{equation}\label{goal:sup:0}
	\lim_{x\to\infty}\sup_{t\in(\tau,T)} u_t(x)=0\qquad\text{a.s.}
\end{equation}
Suppose, for the moment, that we have established \eqref{goal:sup:0}.
Thanks to symmetry,
\eqref{goal:sup:0} also implies that 
$\lim_{x\to-\infty}\sup_{t\in(\tau,T)}u_t(x)=0$ a.s. Because $u$ is almost
surely continuous \cite{Dalang,cbms,Walsh} it follows that
\begin{equation}\label{sup:1}
	\P\left\{ \sup_{x\in\R}u_t(x)<\infty\text{ for all $t\in(\tau\,,T)$}\right\}=1.
\end{equation}
If $\bm{\Lambda}=\infty$, then we can choose $\tau$ as close as we like to $0$
and $T$ as close as we like to $\infty$ in order to deduce Part 1 of
Theorem \ref{th:0-1} from \eqref{sup:1}. Similarly,
if $\bm{\Lambda}<\infty$, then we
can deduce half of Part 3 of Theorem \ref{th:0-1}; specifically, we
can choose $T$ arbitrarily close to $K^2\bm{\Lambda}^3/64$ to see
that $t_1:= K^2\bm{\Lambda}^3/64$ can serve as a candidate for the constant
$t_1$ of Theorem \ref{th:0-1}, Part 3.
We conclude this subsection by verifying \eqref{goal:sup:0}.

Define
\[
	x_n := \sqrt n\qquad\text{for all integers $n\ge 0$},
\]
and
\[
	t(j\,,n) := \frac{jT}{n}\qquad\text{for all
	$j\in \mathds{J}(n\,;\tau\,,T):=\left[
	\frac{n \tau}{T}\,,n\right]\cap\Z$}.
\]
Theorem \ref{th:tails} ensures that for all $\varepsilon>0$ and all
sufficiently-large integers $n\gg1$,
\begin{align*}
	\P\left\{ \max_{j\in \mathds{J}(n;\tau,T)} u_{t(j,n)}(x_n)
		\ge \varepsilon \right\}&\le \sum_{j\in \mathds{J}(n;\tau,T)}
		\P\left\{ u_{t(j,n)}(x_n) \ge \varepsilon\right\}\\
	&\le \text{const}\cdot \sum_{j\in \mathds{J}(n;\tau,T)} 
		\exp\left( -\frac{K\lambda^{3/2}}{\sqrt{t(j\,,n)}} \log|x_n|\right)\\
	&\le \text{const}\cdot 
		\exp\left( -\left[\frac{K\lambda^{3/2}}{2\sqrt{T}}-1\right] \log n\right)
		=O(n^{-3}).
\end{align*}
Therefore, the Borel--Cantelli lemma ensures that
\begin{equation}\label{u:0:1}
	\lim_{n\to\infty} \max_{j\in \mathds{J}(n;\tau,T)} u_{t(j,n)}(x_n)=0
	\qquad\text{a.s.}
\end{equation}

Choose and fix an arbitrary number
$\varrho\in(0\,,\nicefrac14)$, and define $k:=\max(2\,,3/\varrho).$
A standard continuity estimate (see, for example, Walsh
\cite[p.\ 319]{Walsh} and Chen and Dalang
\cite{ChenDalang14Holder}) shows that
\[
	A_{k,\varrho,\tau,T}:=A:=\sup_{x\in\R}
	\E\left(\sup_{\substack{s,t\in(\tau,T):\\
	s\neq t}} \frac{\left| u_s(x) - u_t(x)\right|^k}{|s-t|^{k\varrho}}\right)
	<\infty.
\]
Therefore, for all $\varepsilon>0$ and integers $n\ge 1$,
\begin{align*}
	&\P\left\{ \sup_{t\in(\tau,T)}\min_{j\in \mathds{J}(n;\tau,T)} 
		\left| u_{t(j,n)}(x_n) - u_t(x_n)\right|
		\ge \varepsilon \right\}\\
	&\hskip1in\le \P\left\{ \sup_{\substack{s,t\in(\tau,T):\\
		|s-t| \le T/n}} \left| u_s(x_n) - u_t(x_n)\right|
		\ge \varepsilon \right\}
		\le \frac{AT^k}{\varepsilon^k n^{k\varrho}}  = O(n^{-3}),
\end{align*}
as $n\to\infty$. We may therefore appeal to the Borel--Cantelli lemma
and \eqref{u:0:1} in order to deduce that, with probability one,
\begin{equation}\label{u:0:2}
	\lim_{n\to\infty}\sup_{t\in(\tau,T)} u_t(x_n)=0.
\end{equation}
Let us recall also the following standard continuity estimate 
(see, for example, Walsh \cite[p.\ 319]{Walsh} and
Chen and Dalang \cite{ChenDalang14Holder}):
\[
	B:=B_{k,\varrho,\tau,T} := \sup_{x\in\R}
	\E\left(\sup_{t\in(\tau,T)}\sup_{x< y\le x+1}
	\frac{|u_t(y)-u_t(x)|^k}{|y-x|^{2k\varrho}}\right)<\infty.
\]
Since $x_{n+1} - x_n \le (2n)^{-1/2}$ as $n\to\infty$, it follows that
\[
	\P\left\{ \sup_{t\in(\tau,T)}\sup_{x_n<y\le x_{n+1}}
	|u_t(y)-u_t(x_n)|\ge\varepsilon\right\} \le \frac{B}{\varepsilon^k 
	(2n)^{k\varrho}}
	=O(n^{-3}).
\]
Thanks to the Borel--Cantelli lemma,
the preceding and \eqref{u:0:2} together imply \eqref{goal:sup:0} and conclude this
subsection.

\subsection{Part 2 of the Proof}
We now consider the case that $\bm{\Lambda}<\infty$.
Throughout, we choose and fix three arbitrary numbers: 
\[
	\varepsilon>0;\quad
	\tau> 4L^2\bm{\Lambda}^3;\quad
	\text{and}\quad
	T>\tau;
\]
where $L$ is the constant of Theorem \ref{th:tails}.
Our plan is to prove that
\begin{equation}\label{goal:sup:1}
	\inf_{t\in(\tau,T)}\sup_{x>0} u_t(x)=\infty\quad\text{a.s.}
\end{equation}
If $\bm{\Lambda}>0$, then \eqref{goal:sup:1} implies
that, outside a single $\P$-null set,
$\sup_{x\in\R}u_t(x)=\infty$ for all $t\ge t_2 := 4L^2\bm{\Lambda}^3$.
And if $\bm{\Lambda}=0$, then we choose $\tau$ as close as we would like
to zero in order to see that, outside one $\P$-null set,
$\sup_{x\in\R}u_t(x)=\infty$ for all $t>0$. In other words,
\eqref{goal:sup:1} furnishes proof of the remaining half of
Theorem \ref{th:0-1}.

Before we prove \eqref{goal:sup:1}, we need to recall a few facts 
about parabolic stochastic PDEs. Let 
\[
	u^{(n,0)}_t(x) := (p_t*u_0)(x)\qquad\text{for all $t>0$, $x\in\R$, and
	$n\ge 0$}.
\]
Then iteratively define for each fixed $n\ge 0$, 
\[
	u^{(n,j+1)}_t(x) = (p_t*u_0)(x) + \int_{(0,t)\times
	[x-\sqrt{n t},x+\sqrt{n t}]} p_{t-s}(y-x)
	\sigma\left(u^{(n,j)}_s(y)\right)\xi(\d s\,\d y),
\]
for all $j\ge 0$, $t>0$, and $x\in\R$. We recall the following
result.

\begin{lemma}[Lemma 4.3 of Conus et al \cite{CJK}]\label{lem:CJK:1}
	There exists a finite constant $A$ such that
	for all integers $n\ge 1$
	and real numbers $t>0$,
	\[
		\sup_{x\in\R}
		\E\left(\left|u_t(x) - u^{(n,n)}_t(x)\right|^2\right)
		\le \frac{A\e^{A t- n}}{n^2}.
	\]
\end{lemma}

Actually, Conus et al \cite{CJK}  present a slightly different formulation than
the one that appears above; see Ref.\ \cite[Lemma 10.10]{cbms}
for this particular formulation, as well as proof.

\begin{lemma}[Lemma 4.4 of Conus et al \cite{CJK}]\label{lem:CJK:2}
	Choose and hold fixed an integer $n\ge 1$
	and real numbers $t>0$ and $x_1,\ldots,x_k\in\R$  that satisfy
	$|x_i-x_j| \ge 2n^{3/2}\sqrt{t}$ 
	for all $1\le i\neq j\le k$. Then,
	$\{u^{(n,n)}_t(x_j)\}_{j=1}^k$ are independent.
\end{lemma} 

\begin{proof}[Proof of Theorem \ref{th:0-1}: Part 2.]
	Choose and fix some $\varepsilon>0$,
	and consider the events 
	\[
		E_t(x) := \{\omega\in\Omega:\, u_t(x)(\omega) < \varepsilon\}
		\qquad\text{for every $t,x>0$}.
	\]
	According to Theorem \ref{th:tails}, for every 
	$\lambda\in (\bm{\Lambda}, \tau^{1/3}(4L^2)^{-1/3}]$ 
	we can find a real number $n(\lambda,\varepsilon)> 1$ such that
	\begin{equation}\label{P(E)}
		\P(E_t(x)) \le 1 - x^{-L\lambda^{3/2}/\sqrt{t}}
		\le 1 - x^{-1/2},
	\end{equation}
	uniformly for all $x\ge n(\lambda,\varepsilon)$ and
	$t\in(\tau\,,T)$.
	Consider the events
	\[
		E_t^{(n)}(x) := \left\{\omega:\, u^{(n,n)}_t(x)(\omega) < 2\varepsilon\right\}
		\qquad\text{for $x\in\R$ and $n\ge 1$}.
	\]
	Lemma \ref{lem:CJK:1} ensures the existence of a finite constant
	$c=c(\tau,T,\varepsilon)$ such that
	\[
		\sup_{t\in(\tau,T)}
		\P\left( E_t(x) \setminus E_t^{(n)}(x)\right) \le
		\sup_{t\in(\tau,T)}
		\P\left\{ \left| u_t(x) - u^{(n,n)}_t(x)\right| \ge\varepsilon  \right\}
		\le cn^{-2} \e^{-n},
	\]
	for all integers $n\ge 1$. Therefore,
	\begin{equation}\label{PPP}
		\P\left( \bigcap_{x\in[n^4\,,2n^4]} E_t(x)\right) \le 
		\P\left( \bigcap_{\ell=n^4}^{2n^4} E_t(\ell)\right)\le
		\P\left( \bigcap_{\ell=n^4}^{2n^4} E_t^{(n)}(\ell)\right) + 
		c n^{2}\e^{-n},
	\end{equation}
	uniformly for all integers $n\ge 1$ and real numbers $t\in(\tau\,,T)$.
	Let $x_1:=n^4$ and define iteratively
	\[
		x_{j+1} := x_j + 2 n^{3/2}\sqrt t\qquad\text{for all $j\ge 1$}.
	\] 
	Let
	\[
		\gamma_n :=  \max\left\{ j\ge 1:\, x_j \le 2n^4\right\},
	\]
	and observe that
	\begin{equation}\label{gamma:n}
		\gamma_n \ge \left\lfloor 1 + \frac{n^{5/2}}{2 t^{1/2}}
		\right\rfloor
		\ge \frac{n^{5/2}}{2T^{1/2}}, 
	\end{equation}
	uniformly for all $t\in(\tau\,,T)$ and $n$ sufficiently large.
	Moreover,
	\begin{align*}
		\P\left( \bigcap_{\ell=n^4}^{2n^4} E_t^{(n)}(\ell)\right)  &\le
			\P\left( \bigcap_{j=1}^{\gamma_n} E_t^{(n)}(x_j)\right) =
			\prod_{j=1}^{\gamma_n} \P\left( E_t^{(n)}(x_j)\right)
			&\text{[Lemma \ref{lem:CJK:2}]}\\
		&\le \prod_{j=1}^{\gamma_n}\left[\P(E_t(x_j)) + 
			cn^{2}\e^{-n}\right]
			&\text{[Lemma \ref{lem:CJK:1}]}\\
		&\le \left[ 1 - \frac{1}{\sqrt{2n^4}} + cn^{2}\e^{-n}\right]^{\gamma_n},
	\end{align*}
	uniformly for all $t\in(\tau\,,T)$ and $n$ sufficiently large,
	owing to \eqref{P(E)}. Since $1-y\le \exp(-y)$ for all $y\in\R$,
	the preceding yields
	\[
		\sup_{\tau<t<T}
		\P\left( \bigcap_{\ell=n^4}^{2n^4} E_t^{(n)}(\ell)\right)
		\le \exp\left( - \frac{n^{1/2}}{4 T^{1/2}}\right),
	\]
	for all  sufficiently-large integers $n\gg1$.
	Thanks to \eqref{gamma:n}, the preceding and \eqref{PPP} together yield 
	\begin{equation}\label{P(En)}
		\sup_{t\in(\tau,T)}\P\left\{ 
		\sup_{n^4\le x\le 2n^4} u_t(x) < \varepsilon\right\} \le
		2\exp\left( - \frac{n^{1/2}}{4 T^{1/2}}\right),
	\end{equation}
	for all integers $n$ sufficiently large. Define $t(0):=\tau$,
	and $t(j) := \tau + j(T-\tau)/n$ for all $1\le j\le n$ in order to deduce
	from \eqref{P(En)} that, for every sufficiently-large integer $n$,
	\begin{align}
		&\P\left\{ \inf_{t\in(\tau,T)}\sup_{x\in [n^4,\,2n^4]}
			u_t(x) < \varepsilon \right\}
			\label{Cheb}\\\notag
		&\le \P\left\{ \inf_{0\le j\le n}\sup_{x\in[n^4,\,2n^4]}
			u_{t(j)}(x) < 2\varepsilon \right\} + 
			\P\left\{ \sup_{\substack{s,t\in(\tau,T)\\\notag
			0< t-s < 1/n}}\sup_{x\in
			[n^4,\,2n^4]} \left| u_t(x) - u_s(x) \right|>\varepsilon\right\}\\
			\notag
		&\le 2n\exp\left( - \frac{n^{1/2}}{4 T^{1/2}}
			\right) + \sum_{\nu=1}^{1+\lfloor n^4\rfloor}
			\P\left\{ \sup_{\substack{s,t\in(\tau,T)\\\notag
			0< t-s < 1/n}}\sup_{x\in[\nu,\nu+1]} \left| u_t(x) - u_s(x) \right|>\varepsilon\right\}.
	\end{align}
	A standard modulus of continuity estimate (see, for example, Walsh
	\cite[p.\ 319]{Walsh} and Chen and Dalang \cite{ChenDalang14Holder}) shows that,
	for each fixed $k\ge 1$ and $\varrho\in(0\,,\nicefrac14)$,
	\[
		\sup_{\nu\in\R}
		\E\left(\sup_{\substack{s,t\in(\tau,T)\\
		0< t-s < 1/n}}\sup_{x\in[\nu,\nu+1]} \left| u_t(x) - u_s(x) \right|^k\right)
		\le\text{const}\cdot n^{-k\varrho},
	\]
	for all $n\ge 1$. Let us apply this with $\varrho:=\nicefrac18$ and 
	$k:=64$.
	In this way, we may deduce from \eqref{Cheb} and Chebyshev's inequality that
	\[
		\P\left\{ \inf_{t\in(\tau,T)}\sup_{x>0} u_t(x) <\varepsilon \right\}\le
		\lim_{n\to\infty}\P\left\{ \inf_{t\in(\tau,T)}\sup_{x\in[n^4,\,2n^4]}
		u_t(x) < \varepsilon \right\}=0.
	\]
	Because $\varepsilon>0$ is arbitrary, this proves \eqref{goal:sup:1}.
\end{proof}

\subsection{Part 3 of the Proof}
We now finish the proof of Part 3. Throughout,
$(\Omega\,,\F,\P)$ denotes the underlying probability space,
and we consider only the case that
$0<\bm{\Lambda}<\infty$. 

For every integer $N\ge 1$ consider the stopping time, 
\[
	\mathcal{T}_N:=\inf\left\{ t>0:\ M(t)\geq N\right\}.
\]
Since $\mathcal{T}_N \le \mathcal{T}_{N+1}$ for all $N\ge1$, 
the random variable
\[
	\mathcal{T} :=\lim_{N\rightarrow\infty} \mathcal{T}_N
\]
exists. According to Parts 1 and 2 of the proof of 
Theorem \ref{th:0-1},  
\[
	0<t_1<\mathcal{T}<t_2<\infty\qquad\text{a.s.},
\]
where $t_1$ and $t_2$ are non random and depend only on 
$\bm{\Lambda}$.
In addition, if $t<\mathcal{T}(\omega)$ for some $\omega\in\Omega$, 
then there exists an integer $N(\omega)>0$ 
such that $t\leq\mathcal{T}_{N(\omega)}(\omega)<\mathcal{T}(\omega)$.
This  implies that $M(t)(\omega)\leq N(\omega)<\infty$. 

On the other hand, if $t>\mathcal{T}(\omega_1)$ for some $\omega_1\in\Omega$, 
then there exists some $N_1(\omega_1)>0$ such that $t\geq\mathcal{T}_n(\omega_1)$
for all $n\geq N_1(\omega_1)$.
It follows that  $M(t)(\omega_1)\geq n$ for all $n\geq N_1(\omega_1)$,
whence $M(t)(\omega_1)=\infty$.   This completes the proof of Part 3,
and concludes the proof of Theorem \ref{th:0-1}.\\

\noindent\textbf{Acknowledgements.}
Part of this work were initiated while two of us [L.C. \&\ D.K.] were visiting
{\it Centre de Recerca Matem\`atica} in Barcelona [CRM].
We thank CRM, in particular,
Professors Marta Sanz--Sol\'e, Frederic Utzet, and Josep Vives, for their hospitality
and for providing us with a wonderful research environment.

%\newpage
\begin{small}
\medskip

\noindent\textbf{Le Chen} [\texttt{chenle@ku.edu}].
	Department of Mathematics, University of Kansas,
	Lawrence, KS  66045-7594\\[.2cm]
\noindent\textbf{Davar Khoshnevisan} [\texttt{davar@math.utah.edu}].
	Department of Mathematics, University of Utah,
	Salt Lake City, UT 84112-0090\\[.2cm]
\noindent\textbf{Kunwoo Kim} [\texttt{kunwookim@msri.org}].
	The Mathematical Sciences Research Institute, 17 Gauss Way, Berkeley, CA 94720-5070 
\end{small}

\end{document}